\documentclass{amsart}
\usepackage{amsmath}
\usepackage{amsthm}
\usepackage{amssymb}

\theoremstyle{plain}
\newtheorem{theorem}{Theorem}[section]
\newtheorem{proposition}[theorem]{Proposition}
\newtheorem{corollary}[theorem]{Corollary}
\newtheorem{lemma}[theorem]{Lemma}

\newtheorem{problem}[theorem]{Problem}

\theoremstyle{definition}
\newtheorem{definition}[theorem]{Definition}

\newtheorem{condition}[theorem]{Condition}

\theoremstyle{remark}
\newtheorem{remark}[theorem]{Remark}
\newtheorem{example}[theorem]{Example}

\newcommand{\bA}{\mathbb{A}}

\newcommand{\bN}{\mathbb{N}}

\newcommand{\bZ}{\mathbb{Z}}

\newcommand{\ba}{\pmb{a}}
\newcommand{\bb}{\pmb{b}}
\newcommand{\be}{\pmb{e}}
\newcommand{\bq}{\pmb{q}}
\newcommand{\bv}{\pmb{v}}
\newcommand{\bw}{\pmb{w}}

\newcommand{\cO}{\mathcal{O}}

\newcommand{\cAut}{{\mathcal{A}ut}}
\newcommand{\cAutq}{{\mathcal{A}ut_\rhd}}
\newcommand{\cTrans}{{\mathcal{T}r}}
\newcommand{\bTrans}{{\mathbf{Tr}}}

\newcommand{\As}{{\mathrm{As}}}
\newcommand{\Aut}{{\mathrm{Aut}}}
\newcommand{\Autq}{{\mathrm{Aut}_\rhd}}
\newcommand{\Inn}{{\mathrm{Inn}}}
\newcommand{\Op}{{\mathrm{Op}}}

\newcommand{\SL}{{\mathrm{SL}}}

\newcommand{\Trans}{{\mathrm{Tr}}}
\newcommand{\Hom}{{\mathrm{Hom}}}
\newcommand{\Img}{{\mathrm{Im}}}
\newcommand{\sm}{{\mathrm{sm}}}
\newcommand{\Var}{{\mathrm{Var}}}
\newcommand{\Grp}{{\mathrm{Grp}}}

\title{Quandle varieties, generalized symmetric spaces and $\varphi$-spaces}

\author{Nobuyoshi Takahashi}
\address{
Department of Mathematics, Graduate School of Science, Hiroshima University, 
1-3-1 Kagamiyama, Higashi-Hiroshima, 739-8526 JAPAN}
\email{takahasi@math.sci.hiroshima-u.ac.jp}


\begin{document}

\maketitle

\begin{abstract}
We define a quandle variety as an irreducible algebraic variety $Q$ 
endowed with an algebraically defined quandle operation $\rhd$. 
It can also be seen as an analogue of a generalized affine symmetric space 
or a regular $s$-manifold in algebraic geometry. 

Assume that $Q$ is normal as an algebraic variety 
and that the action of the inner automorphism group 
has a dense orbit. 
Then we show that there is an algebraic group $G$ 
such that each orbit is isomorphic to the quandle $(G/H, \rhd_\varphi)$ 
associated to the group $G$, an automorphism $\varphi$ of $G$ 
and a subgroup $H$ of $G^\varphi$. 
\end{abstract}

\section{Introduction}

\subsection{Generalized symmetric spaces and $\varphi$-spaces}

A symmetric space 
is a Riemannian manifold 
with a symmetry $s_x$ around any point $x$. 
Here, by a symmetry around $x$, we mean an isometry which fixes $x$ 
and inverts tangent vectors at $x$. 
After the extensive studies by \'E. Cartan, 
the theory of symmetric spaces has been significantly generalized. 
One important direction was to regard them 
as manifolds endowed with the binary operation $x\cdot y:=s_x(y)$. 
(See \cite[Introduction and \S\S1--7]{SSS2002} for a survey 
of developments of the theory of generalized symmetric spaces.) 

In \cite{Loos1967} (see also \cite{Loos1969}), 
Loos characterized symmetric spaces in terms of binary operations. 
For the brevity of the presentation, 
let us mean by the word ``automorphism'' 
either an isometry, a diffeomorphism preserving an affine connection 
or just a diffeomorphism according to the context. 
Then a symmetric space is a Riemannian manifold $M$ 
endowed with a smooth binary operation $\cdot: M\times M\to M$ 
subject to the following conditions: 
\begin{enumerate}
\item
$x\cdot x=x$. 
\item
$s_x: M\to M; y\mapsto x\cdot y$ is an automorphism. 
\item
$x\cdot(y\cdot z)=(x\cdot y)\cdot(x\cdot z)$. 
\item
$s_x\circ s_x=id_M$ 
and $x$ is an isolated fixed point of $s_x$. 
\end{enumerate}
In the setting of affine differential geometry, 
the same conditions define an affine symmetric space, 
although Loos proved that 
the affine connection can be canonically constructed from the rest of the data. 

The notion of symmetric space was generalized 
by relaxing the condition (4). 
Loos considered manifolds with an operation satisfying the conditions (1)--(3) and 
\begin{itemize}
\item[(4R)] (reflexivity)
$s_x\circ s_x=id_M$, 
\end{itemize}
and named them reflexion spaces. 

Non-involutive symmetries were introduced to the field by Ledger(\cite{Ledger1967}). 
Then Kowalski and Fedenko 
considered the notions of regular $s$-manifold and generalized affine symmetric space. 
A regular $s$-manifold is defined as a manifold with an operation 
satisfying conditions (1)--(3) and 
\begin{enumerate}
\item[(4T)] (tangential regularity)
$T_xs_x-1$ is invertible on $T_xM$. 
\end{enumerate}
A generalized affine symmetric space 
is an affine manifold 
which \emph{admits} an operation satisfying conditions (1)--(3) and 
\begin{enumerate}
\item[(4I)] (isolatedness)
$x$ is an isolated fixed point of $s_x$. 
\end{enumerate}
In this case, 
the family $\{s_x\}$ of automorphisms 
is called an admissible $s$-structure. 

The conditions (4T) and (4I) are equivalent 
in the presence of a Riemannian metric or an affine connection, 
so a generalized affine symmetric space 
together with an admissible $s$-structure 
is a regular $s$-manifold. 

\medskip
Cartan proved that a symmetric space can be 
described as a special kind of homogeneous space. 
Loos, Fedenko and Kowalski gave generalizations of this result 
in their respective settings. 
In order to state their results, 
let us explain the notion of $\varphi$-space 
introduced by Vedernikov(\cite{Vedernikov1965}). 

\begin{definition}
Let $G$ be a Lie group, $\varphi$ an automorphism\footnote{
In the original definition, an endomorphism is allowed.} of $G$, 
and $H$ a subgroup of $G$ satisfying $(G^\varphi)^\circ \subseteq H\subseteq G^\varphi$. 
Then $G/H$, or the triple $(G, H, \varphi)$, is called a $\varphi$-space. 

It is called a regular\footnote{
There seems to be two meanings for the term ``regular'' here. 
For $s$-manifolds, 
the term regular refers to the condition $x\cdot (y\cdot z)=(x\cdot y)\cdot (x\cdot z)$, 
and for $\varphi$-spaces, it refers to the tangential regularity.} 
$\varphi$-space
if $T_{\bar{e}}\bar{\varphi}-1$ is invertible, 
where $\bar{e}$ is the residue class of the identity element 
and $\bar{\varphi}: G/H\to G/H$ is the induced diffeomorphism. 
\end{definition}

Extending the work of Loos on affine symmetric spaces (see Theorem \ref{theorem_loos}), 
Stepanov, Fedenko and Kowalski studied 
the relation between 
regular $s$-manifolds, generalized affine symmetric spaces and 
regular $\varphi$-spaces. 
Their works show that these three notions 
are essentially equivalent to each other.

\begin{theorem}(\cite[Theorems II. 4 and II.25]{Kowalski1980})
\label{theorem_kowalski}

(1)
If $(G, H, \varphi)$ is a regular $\varphi$-space, 
one can define an operation $\cdot_\varphi$ by 
\[
xH\cdot_\varphi yH := x\varphi(x^{-1}y)H, 
\]
and then $(G/H, \cdot_\varphi)$ is a regular $s$-manifold. 
It also admits a canonical affine connection 
which is invariant under the symmetries, 
so $G/H$ is a generalized affine symmetric space. 

(2)
Conversely, 
for a generalized affine symmetric space 
endowed with an admissible $s$-structure, 
one can construct a regular $\varphi$-space 
which gives the inverse correspondence to the above. 

(3)
Finally, if $M$ is a regular $s$-manifold, 
then there exists uniquely an affine connection on $M$ 
which makes $M$ a generalized affine symmetric space 
with an admissible $s$-structure $\{s_x\}$. 
\end{theorem}

\subsection{Knots and quandles}

To any knot, 
Joyce(\cite{Joyce1982}) and Matveev(\cite{Matveev1982}) 
associated a certain algebraic system, 
which Joyce called the knot quandle. 
A quandle is a nonempty set $Q$ endowed with a binary operation $\rhd$ 
which satisfies the following conditions. 
\begin{enumerate}
\item
$q\rhd q=q$. 
\item
$s_q: Q\to Q; r\mapsto q\rhd r$ is a bijection. 
(We write $q\rhd^{-1}r$ for $s_q^{-1}(r)$.) 
\item
$q\rhd(r\rhd s)=(q\rhd r)\rhd(q\rhd s)$. 
\end{enumerate}
In more descriptive words, 
it is a distributive groupoid. 

To any knot or surface knot is associated its knot quandle 
(or fundamental quandle). 
Given any finite quandle, 
an invariant called the coloring number is defined 
for knots and surface knots. 
There is a theory of quandle homology and cohomology, 
which gives rise to even more invariants(\cite{CJKLS2003}). 

It is obvious that a regular $s$-manifold can be regarded as a quandle. 
Joyce already noted, citing \cite{Loos1967} and \cite{Loos1969}, 
that a reflexion space is a quandle. 
However, 
little attention has been paid 
to regular $s$-manifolds in knot theory, 
probably because finite quandles already provide plenty of useful information. 
Recently, 
Rubinsztein defined the notion of continuous quandle(\cite{Rubinsztein2007}). 
He also defined an invariant topological space $J_Q(L)$ 
for any continuous quandle $Q$ and any link $L$, 
which is a continuous generalization of the coloring number.

\subsection{Quandle varieties}

In this paper, we study quandles and generalized symmetric spaces 
in the category of algebraic varieties. 

By a quandle variety, 
we mean an algebraic variety $Q$ 
equipped with quandle operations $\rhd$ and $\rhd^{-1}$ 
which are regular maps of algebraic varieties. 
Here, an algebraic variety will be always assumed 
to be irreducible and reduced. 

For any quandle $Q$, 
we have two important groups of automorphisms. 
The group $\Inn(Q)$ is defined as the group 
generated by the automorphisms $s_q$. 
The group $\Trans(Q)$ is the subgroup of $\Inn(Q)$ 
consisting of products of the same numbers of $s_q$ and $s_q^{-1}$. 
It can be shown that $\Inn(Q)$-orbits and $\Trans(Q)$-orbits coincide, 
and we say that a quandle is $\rhd$-connected 
if it is comprised of one orbit. 

Our main theorem implies that a $\rhd$-connected quandle variety 
is an algebraic $\varphi$-space $G/H$, 
except that $H$ does not necessarily contain $(G^\varphi)^\circ$. 

\begin{theorem}[See Theorem \ref{theorem_main}]
Let $Q$ be a quandle variety which is irreducible and normal as a variety. 
Assume that $Q$ contains an open $\Inn(Q)$-orbit $O$. 

(1)
The group $\Trans(Q)$ 
can be given a structure of a connected algebraic group $\bTrans(Q)$ 
such that the action of $\bTrans(Q)$ on $Q$ is algebraic. 
If $Q$ is quasi-affine, then $\bTrans(Q)$ is affine. 

(2)
The open orbit $O$ is $\rhd$-connected, 
and there are natural isomorphisms 
$\Trans(O)\cong \Trans(Q)$ and $\Inn(O)\cong\Inn(Q)$.

(3)
For any $q\in Q$, 
the conjugation map $g\mapsto s_q g s_q^{-1}$ 
restricts to an automorphism $\varphi_q$ of $\bTrans(Q)$, 
the stabilizer $\bTrans(Q)_q$ is contained in $\bTrans(Q)^{\varphi_q}$ 
and the natural map $\bTrans(Q)/\bTrans(Q)_q\to\Inn(Q)q;\ g\bTrans(Q)_q\mapsto gq$ 
is an isomorphism of quandle varieties. 

Each orbit acts on another orbit, 
and this action can also be described 
using a certain self-map of $\bTrans(Q)$. 
\end{theorem}

The plan of this paper is as follows. 
We recall the definitions and elementary results 
related to quandles in \S2. 
In \S3, 
we define quandle varieties 
and show several basic results on orbits 
which are similar to results on orbits of an algebraic group action. 
We state and prove the main theorem in \S4. 
The key is to find a space with a $Q$-action 
such that $\Trans(Q)$ can be identified with an orbit. 
In the last section, 
we show a result on the relation 
between isolatedness of fixed points 
and $\rhd$-connectedness, 
which gives a partial answer 
to an algebraic version of a question of Kowalski.

\section{Quandles}

We recall the definition of quandles 
and basic results that will be used in the sequel. 
Most of the contents in this section are from \cite{Joyce1982}. 
In this paper, we use the convention where $Q$ ``acts'' from the left. 

\begin{definition}
A quandle is a nonempty set $Q$ equipped with a binary operation $\rhd$ 
satisfying the following conditions. 
\begin{enumerate}
\item
For any $q\in Q$, $q\rhd q=q$ holds. 
\item
For any $q, r\in Q$, 
there exists a unique element $r'\in Q$ such that $q\rhd r'=r$. 
\item
For any $q, r, s\in Q$, 
$q\rhd(r\rhd s)=(q\rhd r)\rhd(q\rhd s)$. 
\end{enumerate}
We define $s_q: Q\to Q$ by $s_q(r)=q\rhd r$, 
which is bijective by (2). 
We write $q\rhd^{-1} r$ for $s_q^{-1}(r)$. 
For $\sigma_1, \dots, \sigma_n\in\{\pm 1\}$, 
we write $q_n\rhd^{\sigma_n} q_{n-1}\rhd^{\sigma_{n-1}}\cdots\rhd^{\sigma_1} q_0$ 
for $q_n\rhd^{\sigma_n} (q_{n-1}\rhd^{\sigma_{n-1}}(\cdots\rhd^{\sigma_1} q_0))$. 
\end{definition}

\begin{definition}
Let $Q$ and $Q'$ be quandles. 
A homomorphism from $Q$ to $Q'$ is a map $f: Q\to Q'$ 
such that $f(q\rhd r)=f(q)\rhd f(r)$ holds for any $q, r\in Q$. 

An isomorphism is a homomorphism which admits an inverse homomorphism. 
An isomorphism onto itself is called an automorphism. 
\end{definition}

The following is immediate from the definitions. 
\begin{proposition}
(1) 
If $f$ is a homomorphism, $f(q\rhd^{-1}r)=f(q)\rhd^{-1}f(r)$ holds. 

(2)
A homomorphism is an isomorphism if and only if it is bijective. 

(3)
The map $s_q$ is an automorphism for any $q\in Q$. 
\end{proposition}

\begin{example}
If $Q$ is a nonempty set, 
then $q\rhd r:=r$ defines the structure of a trivial quandle. 
\end{example}

\begin{example}\label{ex_ext_of_trivial}
Let $X$ be a nonempty set and $A$ an abelian group. 
Let $F: X\times X\to A$ be a map. 
On $Q=X\times A$, 
the operation $(x, a)\rhd (y, b):=(y, b+F(x, y))$ gives a quandle structure 
if and only if $F(x, x)=0$ for any $x$. 
\end{example}

\begin{example}\label{ex_conj}
If $G$ is a group, 
then $g\rhd h:=g^{-1}hg$ gives a quandle structure on $G$. 
A conjugacy class in $G$ is also a quandle. 
\end{example}

\begin{example}\label{ex_phi_set}
If $G$ is group, $\varphi$ is an automorphism of $G$ 
and $H$ is contained in $G^\varphi$, 
then $xH\rhd_\varphi yH:=x\varphi(x^{-1}y)H$ 
gives a quandle structure on $G/H$. 

In particular, 
if $G=V$ is a vector space,  $\varphi$ is a linear automorphism 
and $H=0$, 
then $\bv\rhd \bw:=\bv+\varphi(\bw-\bv)$ defines a quandle. 
\end{example}

\begin{example}\label{ex_vedernikov}
If $G$ is group and $\varphi$ is an automorphism of $G$, 
then $x\rhd'_\varphi y:=x\varphi(yx^{-1})$ 
gives a quandle structure on $G$. 

This is related to the previous example 
in the following way(\cite{Vedernikov1965}). 
Consider the action of $G$ on $(G, \rhd'_\varphi)$ defined by 
$x\cdot a=xa\varphi(x^{-1})$. 
Then $G/G^\varphi$ can be identified with the orbit of $e$ 
and the quandle operations coincide under this identification. 
More generally, the orbit of $a$ 
can be identified with $(G/G^{\varphi_a}, \rhd_{\varphi_a})$, 
where $\varphi_a(x):=a\varphi(x)a^{-1}$. 
\end{example}

\begin{definition}
We define the automorphism group of $Q$ as 
\[
\Autq(Q):=\{ f: Q\to Q| f(q\rhd r)=f(q)\rhd f(r)\hbox{ and $f$ is bijective}\} 
\]
and the inner automorphism group of $Q$ as 
\[
\Inn(Q):=\langle s_q| q\in Q\rangle. 
\]
Let 
\[
\Trans_k(Q):=
\{s_{b_1}^{-1}\circ\cdots\circ s_{b_k}^{-1}
\circ s_{a_k}\circ\cdots\circ s_{a_1} | a_i, b_i\in Q\} 
\]
and define the group of transvections of $Q$ to be 
\[
\Trans(Q):=\bigcup_{k\in\bN}\Trans_k(Q). 
\]
\end{definition}

We will see that $\Trans(Q)$ is in fact a group. 
Moreover, $\Inn(Q)$ and $\Trans(Q)$ are normal subgroups 
of $\Autq(Q)$. 

\begin{lemma}\label{lem_trans}
(1)
For any $f\in\Autq(Q)$ and $q\in Q$, 
$f\circ s_q=s_{f(q)}\circ f$ holds. 

(2)
For any nonnegative integers $k$ and $l$, 
let $\sigma_1, \dots, \sigma_{k+l}\in\{\pm 1\}$ be such that 
$\#\{i| \sigma_i=1\}=k$ and $\#\{i| \sigma_i=-1\}=l$. 
Define 
\[
S_{\sigma_1, \dots, \sigma_{k+l}}:=
\{s_{a_{k+l}}^{\sigma_{k+l}}\circ\cdots\circ s_{a_1}^{\sigma_1}| a_1, \dots, a_{k+l}\in Q\}. 
\]
Then $S_{\sigma_1, \dots, \sigma_{k+l}}$ depends only on $k$ and $l$. 
\end{lemma}
\begin{proof}
(1)
This follows from 
\[
(f\circ s_q)(r) = f(q\rhd r) 
= f(q)\rhd f(r) 
= (s_{f(q)}\circ f)(r). 
\]

(2)
Since $s_a^{-1}$ is an automorphism, 
it follows from (1) 
that $s_a^{-1}\circ s_b=s_{s_a^{-1}(b)}\circ s_a^{-1}$. 
So $s_{a_i}^{-1}$'s can be gathered to the right. 
\end{proof}

\begin{proposition}
$\Inn(Q)$ and $\Trans(Q)$ are normal subgroups of $\Autq(Q)$. 
\end{proposition}
\begin{proof}
From (2) of the previous lemma, 
we see that $\Trans(Q)$ is a subgroup. 
By (1) of the lemma, 
we have 
\[
f\circ s_{a_k}^{\sigma_k}\circ\cdots\circ s_{a_1}^{\sigma_1}\circ f^{-1}=
s_{f(a_k)}^{\sigma_k}\circ\cdots\circ s_{f(a_1)}^{\sigma_1}, 
\]
so $\Inn(Q)$ and $\Trans(Q)$ are normal subgroups. 
\end{proof}

\begin{definition}
Let $Q$ be a quandle and $X$ a set. 
An action of $Q$ on $X$ is a map $Q\times X\to X; (q, x)\mapsto q\rhd x$ 
subject to the following conditions. 
\begin{enumerate}
\item
For any $q\in Q$ and $x\in X$, 
there exists a unique element $x'\in X$ such that $q\rhd x'=x$. 
\item
For any $q, r\in Q$ and $x\in X$, 
$q\rhd(r\rhd x)=(q\rhd r)\rhd(q\rhd x)$. 
\end{enumerate}
For any $q\in Q$, we define $s_q: X\to X$ by $s_q(x)=q\rhd x$. 
We write $q\rhd^{-1} x$ for $s_q^{-1}(x)$. 

Let $\Op(Q, X)$ be the group generated by $\{s_q| q\in Q\}$. 
Let $\Trans_k(Q, X)$ and $\Trans(Q, X)$ be defined as in the previous definition. 
\end{definition}

We can prove the following 
by the same arguments as in the proof of the previous proposition. 
\begin{proposition}
$\Trans(Q, X)$ is a normal subgroup of $\Op(Q, X)$. 
\end{proposition}

\begin{remark}\label{rem_action} 
(1)
There is a natural action of $Q$ on itself, 
for which 
$\Op(Q, Q)=\Inn(Q)$ and $\Trans(Q, Q)=\Trans(Q)$ hold. 

(2)
Let $\As(Q)$ be the group 
\[
\As(Q):=
\langle
g_q | q\in Q || 
g_q g_r=g_{q\rhd r} g_q 
\rangle. 
\]
Then an action of $Q$ on $X$ is equivalent to 
a group action of $\As(Q)$ on $X$, 
and there is a natural surjective homomorphism 
$\As(Q)\to \Op(Q, X)$. 
Let $\As(Q)_0$ be the subgroup defined by 
\[
\As(Q)_0:=\left\{g_{q_1}^{e_1}\cdots g_{q_k}^{e_k}\left| \sum_{i=1}^k e_i=0\right.\right\}, 
\]
then the above homomorphism restricts to a surjective homomorphism 
$\As(Q)_0\to \Trans(Q, X)$. 

(3)
Note that an action of $Q$ on $X$ does not necessarily 
induce a homomorphism $\Inn(Q)\to \Op(Q, X)$. 
In other words, a $Q$-action on $X$ does not necessarily factor through $\Inn(Q)$. 
For example, let $Q=\bA^1$ be endowed with the trivial operation $q\rhd r=r$. 
Then $Q$ acts on $X=\bA^1$ by $q\rhd x=x+1$. 
This action does not factor through $\Inn(Q)=\{id_Q\}$. 

(4)
Similarly, if we make the trivial quandle $Q=\bA^1$ act on $X=\bA^1$ by $q\rhd x=x+q$, 
then $s_r^{-1}s_q(x)=x+q-r$ 
and so the action of $\Trans(Q, X)$ does not factor through 
$\Trans(Q)=\{id_Q\}$. 
\end{remark}

\begin{proposition}\label{prop_inn_trans_in_q}
The equality $\Inn(Q)q=\Trans(Q)q$ holds 
for any $q\in Q$. 
\end{proposition}
\begin{proof}
It is clear that $\Trans(Q)q$ is contained in $\Inn(Q)q$. 
For the other direction, 
write an element $q'$ of $\Inn(Q)q$ as 
$s_{a_k}^{\sigma_k}\circ\cdots\circ s_{a_1}^{\sigma_1}(q)$. 
Then we have 
\[
q'=s_{q'}^{-\sum \sigma_i}\circ s_{a_k}^{\sigma_k}\circ\cdots\circ s_{a_1}^{\sigma_1}(q) 
\in\Trans(Q)q. 
\]
\end{proof}

\begin{remark}\label{rem_inn_trans_in_x}
For an action of $Q$ on a set $X$, 
the $\Op(Q, X)$-orbits are not necessarily the same as the $\Trans(Q, X)$-orbits. 
In the example of Remark \ref{rem_action}(3), 
we have $\Op(Q, X)0=\bZ$ while $\Trans(Q, X)0=\{0\}$. 
\end{remark}

The study of orbits plays an important role 
in the proof of our main theorem. 
Let us introduce the following notation. 
\begin{definition}
Let $Q$ be a quandle acting on a set $X$ 
and let $Z$ be a subset of $X$. 

For $i\in\bN$, we define $Q^iZ$ inductively by 
$Q^0Z=Z$ and $Q^{i+1}Z=Q\rhd Q^iZ$. 

Similarly, define 
$Q^{-i}Z$ inductively by 
$Q^0Z=Z$ and $Q^{-(i+1)}Z=Q\rhd^{-1} Q^{-i}Z$. 

Define $Q^{\pm i}Z$ by 
$Q^{\pm 0}Z=Z$ and 
$Q^{\pm(i+1)}Z=(Q\rhd Q^{\pm i}Z)\cup(Q\rhd^{-1} Q^{\pm i}Z)$. 
\end{definition}

\begin{proposition}\label{prop_orbit}
(1)
$\Op(Q, X)Z=\bigcup_{i\in\bN} Q^{\pm i}Z$. 

(2)
$\Trans_i(Q, X)\subseteq \Trans_{i+1}(Q, X)$, 
hence 
$\Trans_i(Q, X)Z\subseteq \Trans_{i+1}(Q, X)Z$, 
for any $i\in\bN$. 

(3)
For the natural action of $Q$ on $X=Q$, 
there are inclusions $Q^iZ\subseteq Q^{i+1}Z$, 
$Q^{-i}Z\subseteq Q^{-(i+1)}Z$ and 
$Q^{\pm i}Z\subseteq Q^{\pm(i+1)}Z$. 

(4)
If $f$ is an element of $\Op(Q, X)$, 
then 
$Q^i f(Z)=f(Q^iZ)$, 
$Q^{-i} f(Z)=f(Q^{-i}Z)$ and 
$\Trans_i(Q) f(Z)=f(\Trans_i(Q)Z)$ hold.

(5)
If $Z\subseteq X$ is stable under the action of $Q$, 
i.e. $Q\rhd Z\subseteq Z$ and $Q\rhd^{-1}Z\subseteq Z$, 
then $q\rhd Z=q\rhd^{-1} Z=Z$ holds for any $q\in Q$. 
\end{proposition}
\begin{proof}
(1) is obvious from the definitions. 

(2)
Take an element $q$ of $Q$, and then 
\[
\Trans_{i+1}(Q, X)\supseteq s_q s_q^{-1} \Trans_i(Q, X)=\Trans_i(Q, X)
\]
holds 
by Lemma \ref{lem_trans}. 

(3)
For any $q\in Q^iZ$, we have $q=q\rhd q\in q\rhd Q^iZ\subseteq Q^{i+1}Z$, 
hence $Q^iZ\subseteq Q^{i+1}Z$. 
The assertions about $Q^{-i}Z$ and $Q^{\pm i}Z$ 
can be proven in the same way. 

(4)
Choose a lift of $f$ in $\As(Q)$ 
and let $f_Q$ denote its image in $\Inn(Q)$. 
It follows from the axiom $q\rhd(r\rhd x)=(q\rhd r)\rhd(q\rhd x)$ 
that $f(q\rhd x)=f_Q(q)\rhd f(x)$ and $f(q\rhd^{-1} x)=f_Q(q)\rhd^{-1} f(x)$. 
Since $f_Q$ is an automorphism of $Q$, 
we have the assertion. 

(5)
We have 
$q\rhd Z\subseteq Z$ and 
$q\rhd^{-1} Z\subseteq Z$. 
Applying $s_q^{-1}$ to the former 
and $s_q$ to the latter, 
we obtain the assertion. 
\end{proof}

In this paper, 
we are mainly interested in the situation where the automorphism group is big enough. 

\begin{definition}
A quandle $Q$ is called homogeneous 
if the action of $\Autq(Q)$ on $Q$ is transitive. 

A quandle $Q$ is called $\rhd$-connected 
if the action of $\Inn(Q)$ on $Q$ is transitive. 
By Proposition \ref{prop_inn_trans_in_q}, 
this is equivalent to saying that 
the action of $\Trans(Q)$ on $Q$ is transitive. 
\end{definition}

For a discrete quandle $Q$, 
the homogeneity of $Q$ implies that 
the quandle can be described in terms of a group automorphism 
as in Example \ref{ex_phi_set}. 

\begin{proposition}
\label{prop_homogeneous_quandle}
Let $Q$ be a quandle 
and let $G$ be $\Autq(Q)$, $\Inn(Q)$ or $\Trans(Q)$. 

(1)
For any element $q$ of $Q$, 
let $\varphi_q$ denote the map 
$G\to G; g\mapsto s_q g s_q^{-1}$. 
Then the natural map 
\[
\pi_q: G/G_q\to Gq; \quad gG_q\mapsto gq 
\]
is an isomorphism of quandles 
from $(G/G_q, \rhd_{\varphi_q})$ to $Gq$. 

In particular, 
\begin{itemize}
\item
If $Q$ is homogeneous, 
then it is isomorphic to 
$(\Autq(Q)/\Autq(Q)_q, \rhd_{\varphi_q})$. 
\item
If $Q$ is $\rhd$-connected, 
then it is isomorphic to $(\Inn(Q)/\Inn(Q)_q, \rhd_{\varphi_q})$ 
and $(\Trans(Q)/\Trans(Q)_q, \rhd_{\varphi_q})$. 
\end{itemize}

(2)
For any two elements $q$ and $r$ of $Q$, 
let $\psi_{q, r}$ denote the map $G\to G; g\mapsto s_q g s_r^{-1}$. 
Then there is a well-defined action 
\[
\rhd_{\psi_{q, r}}: G/G_q\times G/G_r\to G/G_r; 
\quad (gG_q, hG_r)\mapsto g\psi_{q, r}(g^{-1}h)G_r 
\]
compatible with the action of $Gq$ on $Gr$. 
\end{proposition}
\begin{proof}
This is essentially Theorems 7.1 and 7.2 of \cite{Joyce1982}. 

First of all, note that $G$ is stable under the map $g\mapsto s_q g s_r^{-1}$. 
This is obvious if $G$ is $\Autq(Q)$ or $\Inn(Q)$. 
If $G$ is $\Trans(Q)$, 
then its element $g$ can be written as 
$s_{a_1}\cdots s_{a_k}s_{b_1}^{-1}\cdots s_{b_k}^{-1}$, 
and so is $s_q g s_r^{-1}$. 

(1)
The natural map $G/G_q\to Gq$ is bijective, 
and the assertion that it is a homomorphism is a special case of (2), 
which we prove below. 
For the second assertion, 
note that $\Trans(Q)q$ is equal to $\Inn(Q)q$ 
by Proposition \ref{prop_inn_trans_in_q}. 

(2)
From $gs_q=s_{g(q)}g$, 
we see that any element of $G_q$ commutes with $s_q$. 
Similarly, any element of $G_r$ commutes with $s_r^{-1}$. 
Thus, if $g'=gg_1$ and $h'=hh_1$ with $g_1\in G_q$ and $h_1\in G_r$, 
we have 
\begin{eqnarray*}
g'\psi_{q, r}((g')^{-1}h')G_r & = & 
gg_1s_qg_1^{-1}g^{-1}hh_1s_r^{-1}G_r \\
& = & gs_qg_1g_1^{-1}g^{-1}hs_r^{-1}h_1G_r \\
& = & gs_qg^{-1}hs_r^{-1}G_r \\
& = & g\psi_{q, r}(g^{-1}h)G_r, 
\end{eqnarray*}
so the action is well-defined. 

We have 
\begin{eqnarray*}
\pi_r(gG_q\rhd_{\psi_{q, r}} hG_r) 
& = & g\psi_{q, r}(g^{-1}h)(r) \\
& = & gs_qg^{-1}hs_r^{-1}(r) \\
& = & s_{g(q)}gg^{-1}h(r) = s_{g(q)}h(r) = g(q)\rhd h(r) \\
& = & \pi_q(gG_q)\rhd \pi_r(hG_r), 
\end{eqnarray*}
so the actions are compatible. 
\end{proof}

\begin{remark}\label{rem_vedernikov_hom}
Let $G$ be $\Autq(Q)$, $\Inn(Q)$ or $\Trans(Q)$. 
For any $q\in Q$, 
the map 
\[
\bar{s}: Q\to (G, \rhd'_{\varphi_q}); x\mapsto s_x s_q^{-1}
\]
is easily seen to be a quandle homomorphism. 

If $Q$ is homogeneous with $G=\Autq(Q)$ or 
$Q$ is $\rhd$-connected with $G=\Inn(Q)$ or $\Trans(Q)$, 
this map is the same as the composite map 
\[
Q\cong (G/G_q, \rhd_{\varphi_q})\twoheadrightarrow (G/G^{\varphi_q}, \rhd_{\varphi_q})
\hookrightarrow (G, \rhd'_{\varphi_q}), 
\]
where the last map is the injective homomorphism 
in Example \ref{ex_vedernikov}. 
\end{remark}

\section{Quandle varieties and orbits of actions}

In this section, we work over an algebraically closed field 
of an arbitrary characteristic.  

\begin{definition}
A quandle variety, or an algebraic quandle, 
is an algebraic (reduced and irreducible) variety $Q$ 
equipped with a quandle operation $\rhd$ 
such that $Q\times Q\to Q\times Q; (q, r)\mapsto (q, q\rhd r)$ 
is an automorphism of a variety. 

An action of $Q$ on an algebraic variety $X$ 
is called an algebraic action 
if $Q\times X\to Q\times X; (q, x)\mapsto (q, q\rhd x)$ 
is an automorphism of a variety. 
A variety with an action of $Q$ will be called 
a $Q$-variety. 
\end{definition}
In the sequel, 
actions will be always algebraic 
when we are dealing with quandle varieties and algebraic varieties. 

\begin{example}
(1)
In Example \ref{ex_ext_of_trivial}, 
if $X$ is an algebraic variety, 
$A$ is a connected commutative algebraic group 
and $F$ is a regular map, 
then $X\times A$ is a quandle variety. 

(2)
If $G$ is a connected algebraic group 
in Example \ref{ex_conj}, 
then the variety $G$ with the conjugation operation is a quandle variety. 

(3)
In Example \ref{ex_phi_set}, 
we have a quandle variety $(G/H, \rhd_\varphi)$ if $G$ is a connected algebraic group, 
$\varphi$ is an algebraic automorphism and $H$ is a closed subgroup. 

(4)
In Example \ref{ex_vedernikov}, 
we have a quandle variety $(G, \rhd'_\varphi)$ 
if $G$ is a connected algebraic group 
and $\varphi$ is an algebraic automorphism. 
The quandle $(G/G^\varphi, \rhd_\varphi)$ 
from the previous example 
can be embedded into this quandle. 
\end{example}

\begin{definition}
Let $G$ be an algebraic group, $\varphi$ an automorphism of $G$, 
and $H$ a subgroup of $G$ satisfying $H\subseteq G^\varphi$. 
Then $(G/H, \rhd_\varphi)$, or the triple $(G, H, \varphi)$, 
is called a weak algebraic $\varphi$-space. 

If $H$ contains the connected component $(G^\varphi)^\circ$ 
of $G^\varphi$ containing the identity element, 
then $G/H$ is called an algebraic $\varphi$-space. 

We say that $G/H$ is a regular algebraic $\varphi$-space  
if $T_{\bar{e}}\bar{\varphi}-1$ is invertible, 
where $\bar{e}$ is the residue class of the identity element 
and $\bar{\varphi}: G/H\to G/H$ is the induced automorphism. 
\end{definition}

The group $\Trans(Q)$ plays an important role 
in the theory of generalized symmetric spaces. 
A key fact is that it is a connected Lie group. 
Now we will show a partial analogue of this fact for quandle varieties. 
First, we prove that orbits of $\Trans(Q)$-actions 
enjoy a number of properties 
similar to those satisfied by orbits of algebraic group actions. 

Let $Q$ be a quandle variety acting on a variety $X$. 
Recall from Proposition \ref{prop_orbit}(2) 
that $\{\Trans_i(Q, X)Z\}$ is an increasing sequence. 

\begin{lemma}
(1) 
Let $Z$ be an irreducible subset of $X$. 
If $i_0$ is a natural number such that 
$\dim \Trans_{i_0}(Q, X)Z=\dim \Trans_{i_0+1}(Q, X)Z$, 
then $\overline{\Trans_i(Q, X)Z}$ is equal to $\overline{\Trans(Q, X)Z}$ 
for $i\geq i_0$. 

In particular, 
$\overline{\Trans_i(Q, X)Z}$ is equal 
to $\overline{\Trans(Q, X)Z}$ for $i\geq \dim X-\dim Z$. 

(2)
Let $Z$ be a constructible subset of $X$. 
Then $\overline{\Trans_i(Q, X)Z}$ is equal 
to $\overline{\Trans(Q, X)Z}$ for $i\geq \dim X$. 
\end{lemma}

\begin{proof}
(1)
Since $\{\overline{\Trans_i(Q, X)Z}\}$ is an 
increasing sequence of irreducible closed sets, 
the equality of dimensions implies 
$\overline{\Trans_{i_0}(Q, X)Z}=\overline{\Trans_{i_0+1}(Q, X)Z}$. 
Let $W$ denote this set. 
It is obviously contained in $\overline{\Trans(Q, X)Z}$. 

Since 
\[
s_q^{-1}s_rW\subseteq \overline{s_q^{-1}s_r\Trans_{i_0}(Q, X)Z}
\subseteq \overline{\Trans_{i_0+1}(Q, X)Z}
= W
\]
holds for any $q, r\in Q$, 
the set $W$ is stable under the action of $\Trans(Q, X)$, 
hence contains $\Trans(Q, X)Z$. 
Since it is closed, it contains $\overline{\Trans(Q, X)Z}$. 

(2) 
This is an immediate consequence of (1). 
\end{proof}

\begin{proposition}\label{prop_trans_orbit_in_x}
Let $z_0$ be a point of $X$. 
Then $\Trans(Q, X)z_0$ is a locally closed subvariety of $X$. 

If $i_0$ is a natural number such that 
$\dim \Trans_{i_0}(Q, X)z_0=\dim \Trans_{i_0+1}(Q, X)z_0$, 
then $\Trans_i(Q, X)z_0=\Trans(Q, X)z_0$ holds for $i\geq 2i_0$. 

In particular, $\Trans_i(Q, X)z_0=\Trans(Q, X)z_0$ holds for $i\geq 2\dim X$. 
\end{proposition}

\begin{proof}
Let $W=\overline{\Trans(Q, X)z_0}$. 
For a subset $V$ of $W$, let $V^\circ$ denote the relative interior of $V$ in $W$. 

By the previous lemma and Chevalley's theorem on the image of a morphism, 
$\Trans_{i_0}(Q, X)z_0$ contains a nonempty open subset of $W$. 
In general, if a group acts on a topological space by homeomorphisms 
and an orbit contains a nonempty open subset, 
then it is easily seen to be open. 
Thus $\Trans(Q, X)z_0$ is open in $W$, 
hence locally closed in $X$. 

Let $z_1$ be an element of $\Trans(Q, X)z_0$ 
and write $z_1$ as $f(z_0)$ with $f\in\Trans(Q, X)$. 
By Proposition \ref{prop_orbit} (4), 
$\Trans_{i_0}(Q, X)z_1$ is equal to $f(\Trans_{i_0}(Q, X)z_0)$, 
which is dense in $W$. 
It follows that 
$\Trans_{i_0}(Q, X)z_0$ and $\Trans_{i_0}(Q, X)z_1$ have nonempty intersection 
and therefore that $z_1$ is contained in $\Trans_{2i_0}(Q, X)z_0$. 
\end{proof}

Thus, 
we see that $X$ is divided into locally closed $\Trans(Q, X)$-orbits.

\medskip
Recall that an open subgroup of a connected algebraic group 
is the whole group. 
As an analogy, for an open subquandle $U$ of $Q$, 
it is hoped that $\Inn(U)\cong\Inn(Q)$ and $\Trans(U)\cong\Trans(Q)$. 
The main theorem asserts that this is true if $Q$ has an open orbit. 
In general, the following holds. 

\begin{proposition}\label{prop_inn_trans_opensub}
Let $Q$ be an arbitrary quandle variety 
and $U$ an open subquandle. 
Then 
there are natural injective homomorphisms 
$\Inn(U)\to\Inn(Q)$ and $\Trans(U)\to\Trans(Q)$. 
\end{proposition}
\begin{proof}
If we write an element of $\Inn(U)$ 
as a product of $s_q^{\pm 1}$'s, 
then the same expression defines an element of $\Inn(Q)$, 
and this gives a well-defined homomorphism $\Inn(U)\to\Inn(Q)$ 
since $U$ is dense in $Q$. 
It is injective since $U\subseteq Q$. 
By restriction, we obtain an injective homomorphism 
$\Trans(U)\to\Trans(Q)$. 
\end{proof}

The following is a kind of approximation of 
$\Inn(U)\cong\Inn(Q)$ and $\Trans(U)\cong\Trans(Q)$. 
\begin{proposition}\label{prop_orbit_opensub}
Let $Q$ be a quandle variety, 
$U$ an open subquandle 
and $X$ a $Q$-variety. 
Then $\Op(U, X)z_0=\Op(Q, X)z_0$ 
and $\Trans(U, X)z_0=\Trans(Q, X)z_0$ hold. 
\end{proposition}

\begin{proof}
Write $G$ for $\Trans(Q, X)$ and $H$ for $\Trans(U, X)$. 
Since $U$ is dense in $Q$, 
$\Trans_i(U, X)z_0$ is dense in $\Trans_i(Q, X)z_0$. 
For $i\gg 0$, 
they are equal to $Hz_0$ and $Gz_0$, respectively, 
and are locally closed in $X$ 
by Proposition \ref{prop_trans_orbit_in_x}. 
So $Hz_0$ is open in $Gz_0$. 
If we take a point $z$ of $Gz_0$, 
then $Hz$ is open in $Gz$ by the same arguments, 
and the latter is equal to $Gz_0$. 
Thus $Hz_0$ and $Hz$ intersect, 
and $z$ is contained in $Hz_0$. 

Choose a point $q\in U$, 
then $\Op(Q, X)$ is equal to $\bigcup_{n\in\bZ}\Trans(Q, X)s_q^n$. 
Thus we have $\Op(U, X)z_0=\Op(Q, X)z_0$. 
\end{proof}

\begin{corollary}
A $\rhd$-connected quandle variety $Q$ 
contains no nontrivial open subquandle. 
\end{corollary}
\begin{proof}
Let $U$ be an open subquandle. 
For any $q\in U$, 
we have 
$U\supseteq \Inn(U)q=\Inn(Q)q=Q$ by the proposition. 
\end{proof}

\medskip
For the natural $Q$-action on $Q$ itself, 
the $\Inn(Q)$-orbits are equal to the $\Trans(Q)$-orbits 
by Proposition \ref{prop_inn_trans_in_q}. 
They are also equal to the ``forward'' orbits 
in the case of quandle varieties. 

\begin{lemma}\label{lem_closed_orbit_in_q}
(1) 
Let $Z$ be an irreducible subset of $X$. 
If $i_0$ is a natural number such that 
$\dim Q^{i_0}Z=\dim Q^{i_0+1}Z$, 
then $\overline{Q^iZ}$ is equal 
to $\overline{\Inn(Q)Z}$ 
for $i\geq i_0$. 

In particular, 
$\overline{Q^iZ}$ is equal 
to $\overline{\Inn(Q)Z}$ for $i\geq \dim X-\dim Z$. 

(2)
Let $Z$ be a constructible subset of $X$. 
Then $\overline{Q^iZ}$ is equal 
to $\overline{\Inn(Q)Z}$ for $i\geq \dim X$. 
\end{lemma}
By symmetry of $\rhd$ and $\rhd^{-1}$, 
similar statements hold for $\overline{Q^{-i}Z}$. 

\begin{proof}
(1)
Using Proposition \ref{prop_orbit} (3) this time, 
the equality of dimensions implies 
that $\overline{Q^iZ}$ stabilizes to an irreducible closed set $W$ for $i\geq i_0$. 
Obviously we have $W\subseteq \overline{\Inn(Q)Z}$. 

For any $q\in W$, 
we have 
\[
q\rhd W = q\rhd \overline{Q^{i_0}Z} 
\subseteq \overline{q\rhd Q^{i_0}Z} 
\subseteq \overline{Q^{i_0+1}Z} 
= W. 
\]
Since $q\rhd W$ is an irreducible closed set with $\dim (q\rhd W)=\dim W$, 
the equality $q\rhd W=W$ holds. 
Applying $s_q^{-1}$, 
we also have $W=q\rhd^{-1} W$. 
Thus the set $W$ is stable under the action of $Q$, 
hence contains $\Inn(Q)Z$. 
Since it is closed, it contains $\overline{\Inn(Q)Z}$. 

(2) is immediate from (1). 
\end{proof}

\begin{proposition}\label{prop_orbit_in_q}
Let $z_0$ be a point of $Q$. 
Then $\Inn(Q)z_0$ is a locally closed subvariety of $Q$. 

If $i_1$ and $i_2$ are natural numbers for which 
$\dim Q^{i_1}z_0=\dim Q^{i_1+1}z_0$ and \linebreak
$\dim Q^{-i_2}z_0=\dim Q^{-(i_2+1)}z_0$ hold, 
then $Q^iz_0=Q^{-i}z_0=\Inn(Q)z_0$ holds for $i\geq i_1+i_2$. 

In particular, $Q^iz_0=Q^{-i}z_0=\Inn(Q)z_0$ holds for $i\geq 2\dim Q$. 
\end{proposition}
\begin{proof}
We write $W=\overline{\Inn(Q)z_0}$ 
and denote the relative interior of $V\subseteq W$ by $V^\circ$. 

By the previous lemma and Chevalley's theorem, 
$Q^{i_1}z_0$ contains a nonempty open subset of $W$. 
Therefore $\Inn(Q)z_0$ is open in $W$, 
hence locally closed in $Q$. 

If $z_1=f(z_0)$ with $f\in\Inn(Q)$, 
then $Q^{-i_2}z_1$ is equal to $f(Q^{-i_2}z_0)$ 
by Proposition \ref{prop_orbit} (4), 
so it is dense in $W$. 
It follows that 
$Q^{i_1}z_0$ and $Q^{-i_2}z_1$ have nonempty intersection 
and therefore that $z_1$ is contained in $Q^{i_1+i_2}z_0$. 
\end{proof}

\begin{remark}\label{rem_orbit_in_x}
For an action on an arbitrary $X$, 
the example in Remark \ref{rem_inn_trans_in_x} 
shows that $\Op(Q, X)z_0$ is not necessarily constructible. 
\end{remark}

\begin{remark}
For an arbitrary irreducible closed subset $Z$ of $Q$, 
it is not necessarily true that $\Inn(Q)Z$ is locally closed. 
For example, let $Q=\SL_2(k)\times \bA^1$ be endowed with 
the operation $(h, t)\rhd (g, s)=(h^{-1}gh, s)$. 
Let 
$Z=\left\{\left.(\left(
\begin{array}{cc}1 & s \\ 0 & 1\end{array}
\right), s)\right| s\in\bA^1\right\}$. 
Then the fiber of $\Inn(Q)Z$ over $s\in \bA^1$ is 
$\{A\in\SL_2(k) | (A-I)^2=O, A\not=I\}$ if $s\not=0$ 
and is $\{I\}$ if $s=0$. 
It follows that $\overline{\Inn(Q)Z}=\{A\in\SL_2(k) | (A-I)^2=O\}\times\bA^1$ and 
that $\Inn(Q)Z$ is not open in its closure. 
\end{remark}

In showing $\rhd$-connectedness of a quandle, 
the following proposition is useful. 

\begin{proposition}\label{prop_connectedness_criteria}
For a quandle variety, the following conditions are equivalent. 
\begin{enumerate}
\item[(1)]
$Q$ is $\rhd$-connected. 
\item[(2)]
$Q^iq=Q$ for any $q\in Q$ and $i\geq 2\dim Q$. 
\item[(2')]
$Q^iq=Q$ for some $q\in Q$ and $i\geq 2\dim Q$. 
\item[(3)]
$Q^iq=Q$ for any $q\in Q$ and $i\gg 0$. 
\item[(3')]
$Q^iq=Q$ for some $q\in Q$ and $i\gg 0$. 
\item[(4)]
$\overline{Q^iq}=Q$ 
for any $q\in Q$ and $i\geq \dim Q$. 
\item[(5)]
$\overline{Q^iq}=Q$ 
for any $q\in Q$ and $i\gg 0$. 
\end{enumerate}
\end{proposition}
\begin{proof}
Assume the condition (1). 
Then (2) and (4) follow from the previous proposition and lemma. 

The implications from (2) to (2')--(3'), 
from (2)--(3') to (1) and from (4) to (5) are obvious. 

Assume that (5) holds. 
Then for any $q, q'\in Q$, there exists $i$ such that 
$Q^iq$ and $Q^iq'$ are both dense. 
By Chevalley's lemma, they contain nonempty open subsets, 
so they have nonempty intersection. 
Thus $q$ and $q'$ can be connected by an inner automorphism. 
\end{proof}

\begin{corollary}
If $Q$ is nonsingular and $T_qs_q-1$ is invertible for any $q\in Q$, 
then $Q$ is $\rhd$-connected. 
\end{corollary}
\begin{proof}
Let $t_q: Q\to Q$ be defined by $t_q(r)=r\rhd q$. 
From $r\rhd r=r$, we see that $T_qt_q+T_qs_q=1$. 
Thus the assumption means that $T_qt_q$ is invertible. 
It follows that $t_q$ is dominant, 
and $Q$ is $\rhd$-connected 
by the proposition. 
\end{proof}

\begin{remark}\label{rem_unip}
The converse does not hold. 
Consider the quandle 
\[
Q=\left\{\left. A^{-1}\left(\begin{array}{cc}
1 & 1 \\ 0 & 1
\end{array}\right)A\right| A\in\SL_2(k)\right\} 
\]
with the conjugation operation. 

We show that $Q$ is $\rhd$-connected. 
Note that an element of $Q$ can be written as 
$M=\left(\begin{array}{cc}
x & y \\ z & w
\end{array}\right)$ 
with $xw-yz=1$, $x+w=2$ and $M\not=I_2$. 
Let $J_t$ denote the matrix $\left(\begin{array}{cc}
1 & t \\ 0 & 1
\end{array}\right)$. 
Then we have $M\rhd J_1=\left(\begin{array}{cc}
1+zw & w^2 \\ -z^2 & 1-zw
\end{array}\right)$. 
Since $z$ and $w$ can be considered as algebraically independent functions, 
we see that $Q\rhd J_1$ is dense. 
Using the transitive action of $\SL_2(k)$, 
we see that $Q\rhd M$ is dense for any $M\in Q$. 
Now Proposition \ref{prop_connectedness_criteria} 
tells us that $Q$ is $\rhd$-connected. 

Alternatively, 
note that the $\Inn(Q)$-orbits are equal to the orbits by the conjugacy action 
of the group generated by $Q$. 
We see that $Q\cdot\{J_t| t\not=0\}$, and therefore $G$, 
contains a nonempty open subset of $\SL_2(k)$ 
by an argument similar to the above. 
Such a subgroup must be equal to $\SL_2(k)$, 
hence the orbit of $J_1$ is equal to $Q$. 

On the other hand, 
the fixed point locus of $s_{J_1}$ is 
$\{J_t| t\not=0\}$. 
Again using the $\SL_2(k)$-action, 
we see that the fixed point locus of $s_M$ is a $1$-dimensional set 
containing $M$ for any $M\in Q$. 
In particular, $T_Ms_M-1$ is singular. 
\end{remark}

\begin{corollary}\label{cor_open_orbit_is_connected}
If an orbit $U:=\Inn(Q)q_0$ is dense in $Q$, 
then it is open 
and is a $\rhd$-connected quandle. 
\end{corollary}
\begin{proof}
Openness follows from Proposition \ref{prop_orbit_in_q}. 
For any $q\in U$, we have $Q^iq=U$ for some $i$ 
by the same proposition. 
Hence $U^iq$ is dense in $U$, 
and $U$ is $\rhd$-connected by the previous proposition. 
\end{proof}

\begin{remark}
In general, an $\Inn(Q)$-orbit is not necessarily $\rhd$-connected. 
In Example \ref{ex_ext_of_trivial}, 
let $X=A=\bA^1$ and $F(x, y)=y-x$. 
Then $Q=\bA^2$ with $(x, a)\rhd (y, b)=(y, b+y-x)$, 
and the orbit of $(0, 0)$ is $\{0\}\times \bA^1$, 
which is a trivial quandle. 
\end{remark}

\section{Main Theorem}

In this section, 
we work over an algebraically closed field $k$ of characteristic $0$. 
The word ``variety'' will always mean an irreducible and reduced variety. 

We denote the category of algebraic varieties over $k$ by $(\Var_k)$ 
and the category of groups by $(\Grp)$. 
For a quandle variety $Q$, let 
$\cTrans(Q): (\Var_k)\to (\Grp)$ 
be the functor defined by 
\[
\cTrans(Q)(S) :=
\{
f\in\Aut_S(S\times Q); 
\hbox{$f|_{\{s\}\times Q}\in\Trans(Q)$ 
for any $s\in S(k)$} 
\}. 
\]

The following is our main theorem. 

\begin{theorem}\label{theorem_main}
Let $Q$ be a quandle variety which is irreducible and normal as a variety. 
Assume that $Q$ contains an open $\Inn(Q)$-orbit $O$. 

(1)
The group $\Trans(Q)$ can be given a structure of a connected algebraic group 
$\bTrans(Q)$ 
such that the action map 
$\alpha: \bTrans(Q)\times Q \to \bTrans(Q)\times Q; (f, q)\mapsto (f, f(q))$ 
is a regular map. 
If $Q$ is quasi-affine, then $\bTrans(Q)$ is affine. 
The pair $(\bTrans(Q), \alpha)$ represents the functor $\cTrans(Q)$. 

(2)
The open orbit $O$ is $\rhd$-connected, 
and there are natural isomorphisms 
$\Trans(O)\cong \Trans(Q)$ and $\Inn(O)\cong\Inn(Q)$.

(3)
For any two elements $q$ and $r$ of $Q$, 
let 
\[
\varphi_q: \bTrans(Q)\to\bTrans(Q); g\mapsto s_q gs_q^{-1}
\]
and 
\[
\psi_{q, r}: \bTrans(Q)\to\bTrans(Q); g\mapsto s_q gs_r^{-1}
\]
be the maps from Proposition \ref{prop_homogeneous_quandle}. 
Then $\varphi_q$ is an automorphism of an algebraic group 
and $\psi_{q, r}$ is an automorphism of an algebraic variety. 

The natural map 
\begin{eqnarray*}
\pi_q: \bTrans(Q)/\bTrans(Q)_q & \to & \Inn(Q)q \\
 g\bTrans(Q)_q & \mapsto & gq
\end{eqnarray*}
of Proposition \ref{prop_homogeneous_quandle} 
is an isomorphism of quandle varieties 
compatible with actions between orbits. 
\end{theorem}

Before starting the proof of the theorem, 
let us recall how the corresponding fact 
in differential geometry was proven (see Theorem \ref{theorem_kowalski}). 
To prove that a regular $s$-manifold $M$ is 
a regular $\varphi$-space, 
one first constructs a canonical affine connection 
which turns $M$ into a generalized affine symmetric space. 
This rigidifies the manifold $M$, 
in the sense that a diffeomorphism which preserves the affine connection 
is determined by the image of a point 
and the tangent map there. 
So $\Trans(M)$ can be realized as a subset of a finite dimensional space. 
Then, it can be proven that $\Trans(M)$ is a Lie group 
and that $M$ is a homogeneous space. 

Accordingly, our proof is divided into two parts. 
First, we find certain structure which rigidifies the space 
and embed $\Trans(Q)$ into an algebraic variety. 
Second, we show that the image of this embedding 
is an algebraic group satisfying the conditions. 

In view of Proposition \ref{prop_homogeneous_quandle},
one might hope to extend our result to quandle varieties 
which are homogeneous in the sense that 
the action of $\Autq(Q)$ is transitive. 
However, it seems difficult to find a good algebraic subgroup 
of $\Autq(Q)$ 
without any assumption on the action of $\Inn(Q)$. 
For example, consider $Q=\bA^n$ with the trivial quandle operation. 
Then the group $\Autq(Q)$ is equal to $\Aut(\bA^n)$, 
which is a very complicated group. 
\medskip

To carry out our plan of proof, 
we consider the following condition for a $Q$-variety $X$. 

\begin{condition}\label{condition_good_action}
(1)
There exists a natural transformation $\Phi$ 
from $\cTrans(Q)$ to 
\[
\cAut(X): (\Var_k)\to (\Grp); 
S\mapsto 
\Aut_S(S\times X) 
\]
such that $\Phi(pt)$ commutes with 
the natural surjective homomorphisms $\As(Q)_0\to \Trans(Q)$ 
and $\As(Q)_0\to \Trans(Q, X)$. 

(2)
There exists a point $x\in X$ such that $\Trans(Q)_x=\{id\}$. 
\end{condition}

\begin{remark}
By condition (1), 
we have a natural homomorphism $\Phi(pt): \Trans(Q)\to \Trans(Q, X)$. 
It is an isomorphism by condition (2). 

The natural transformation $\Phi$ is unique if it exists. 
In fact, let $f$ be a family of transvections over $S$. 
Then $\Phi(S)(f)$ must be $(s, x)\mapsto (s, (\Phi(pt)(f_s))(x))$. 
\end{remark}

The theorem can be reduced to the following propositions. 

\begin{proposition}\label{prop_good_action}
Let $Q$ be a quandle variety 
and assume that $Q$ contains an open $\Inn(Q)$-orbit $O$. 
Let $n$ be a positive integer and 
$X=\Hom_{O^n}(k^{\dim Q}\times O^n, TO^n)\times O$, 
where $\Hom$ denotes the vector bundle of homomorphisms over $O^n$. 
Then $X$ satisfies Condition \ref{condition_good_action} for $n\gg 0$. 
More precisely, the following hold. 

(1)
For any $n$, 
there is a natural transformation from $\cAutq(Q)$ to $\cAut(X)$. 
Hence $\Autq(Q)$ acts on $X$, 
and $Q$ acts algebraically on $X$ via $\Inn(Q)$. 

(2)
For $n\gg 0$, 
there exists a point $x\in X$ 
such that $\Autq(Q)_x=\{id\}$. 
\end{proposition}

\begin{proposition}\label{prop_good_action_to_algebraicity}
Let $Q$ be a normal quandle variety. 
Let $X$ be a $Q$-variety 
satisfying Condition \ref{condition_good_action} 
with a point $x$ as in the condition. 
Note that 
$\bTrans(Q):=\Trans(Q)x$ is a locally closed subvariety of $X$ 
by Proposition \ref{prop_trans_orbit_in_x}. 

(1)
If $U$ is an open subquandle, 
then there are natural isomorphisms 
$\Trans(U)\cong \Trans(Q)$ and $\Inn(U)\cong\Inn(Q)$. 

(2)
The action map $\alpha: \bTrans(Q)\times Q\to \bTrans(Q)\times Q$ 
is an automorphism of a variety. 

(3)
The pair $(\bTrans(Q), \alpha)$ represents the functor $\cTrans(Q)$. 

(4)
Under the identification of $\bTrans(Q)$ and $\Trans(Q)$ by the natural bijection, 
$\bTrans(Q)$ with the induced multiplication is an algebraic group. 

(5)
For any $q, r\in Q$, 
the map $\psi_{q, r}: \bTrans(Q)\to \bTrans(Q); g\mapsto s_q g s_r^{-1}$ 
is an automorphism of $\bTrans(Q)$ as an algebraic variety. 
\end{proposition}

\begin{proof}[
Proposition \ref{prop_good_action}, \ref{prop_good_action_to_algebraicity} 
$\Rightarrow$ Theorem \ref{theorem_main}]
Assertions (1) and (2) are immediate consequences of the propositions 
and Corollary \ref{cor_open_orbit_is_connected}. 

It also follows that $\bTrans(Q)/\bTrans(Q)_q$ 
together with the operation induced by $\varphi_q$ 
is a quandle variety. 
Since the natural map $\bTrans(Q)/\bTrans(Q)_q\to \Inn(Q)q$ 
is a one-to-one regular map of nonsingular varieties 
over an algebraically closed field of characteristic $0$, 
it is an isomorphism of varieties. 
It is an isomorphism of quandles compatible with actions between orbits 
by Proposition \ref{prop_homogeneous_quandle}. 
Thus (3) holds. 
\end{proof}

Now let us prove the propositions. 

\begin{proof}[Proof of Proposition \ref{prop_good_action}]
(1) 
Note that $\Hom_{O^n}(k^{\dim Q}\times O^n, TO^n)$ 
can be identified with the fiber product of $\dim Q$ copies of $TO^n$ 
over $O^n$. 
Then the first assertion follows from the following functorial constructions. 
\begin{enumerate}
\item[(a)]
A family of quandle automorphisms on $Q\times S$ 
induces an automorphism of $O\times S$ over $S$. 
\item[(b)]
Let $f$ and $g$ be automorphisms of $V\times S$ and $W\times S$, 
respectively, over $S$. 
Then $f\times_S g$ is an automorphism of 
$(V\times S)\times_S(W\times S)\cong(V\times W)\times S$ 
over $S$. 
\item[(c)]
If $f$ is an automorphism of $V\times S$ with $V$ nonsingular, 
then $T_{V\times S/S}f$ is an automorphism of $TV\times S$ 
which commutes with $f$. 
\item[(d)]
Generalizing (b), 
if $V$ and $W$ come with morphisms to $Y$, 
and $f$ and $g$ are automorphisms of $V\times S$ and $W\times S$ over $S$ 
commuting with an automorphism $h$ of $Y\times S$ over $S$, 
then $f\times_S g$ restricts to an automorphism 
of $(V\times S)\times_{Y\times S} (W\times S)\cong(V\times_Y W)\times S$. 
\end{enumerate}

The second assertion follows from the first assertion. 
Note that the action of $Q$ on $Q$ gives an element of $\cAutq(Q)(Q)$, 
hence an element of $\cAut(X)(Q)$, i.e. 
an automorphism of $Q\times X$ over $Q$. 
This gives a family $\{s_{q, X}\}_{q\in Q}$ of automorphisms of $X$. 
The homomorphism $\Phi(pt): \Autq(Q)\to \Aut(X)$ 
sends $s_q$ to $s_{q, X}$, 
as we see by mapping a point to $q$ 
and using the naturality of $\Phi$. 
Thus we have an algebraic action of $Q$ on $X$ 
which factors through $\Autq(Q)$, 
hence through $\Inn(Q)$. 

(2)
Let 
\[
\mu: Q^n\times Q\to Q;\ ((q_1, \dots, q_n), q)\mapsto q_n\rhd\cdots\rhd q_1\rhd q 
\] 
be the multiplication map. 
Let $q_0$ be a point of $O$. 
By Proposition \ref{prop_orbit_in_q} 
and the assumption that $O$ is an open orbit, 
we see that 
\[
\mu_1: Q^n\to O;\ \bq\mapsto \mu(\bq, q_0)
\] 
is surjective for $n\gg 0$. 
Since we are working over a field of characteristic $0$, 
there exists a point $\bq_0\in O^n$ 
such that $T_{\bq_0}\mu_1$ is surjective. 
Define $\mu_2$ by 
\[
\mu_2: Q\to Q;\ q\mapsto \mu(\bq_0, q). 
\]
Let $q':=\mu(\bq_0, q_0)=\mu_1(\bq_0)=\mu_2(q_0)$, 
and choose a linear map $\varphi: k^d\to T_{\bq_0} O^n$ 
such that $T_{\bq_0}\mu_1\circ\varphi: k^d\to T_{q'}O$ is an isomorphism. 
Then $x:=((\bq_0, \varphi), q_0)$ defines a point of $X$. 
Now we claim that $\Autq(Q)_x$ is trivial. 

To show the claim, 
we first associate a vector field $\tilde{\bv}$ on $O$ 
to each $\bv\in k^d$. 
This is what ``rigidifies'' $Q$. 
Informally, this is defined by 
\[
\mu(\bq_0 + \epsilon \varphi(\bv), q)
=
\mu(\bq_0, q)
+ \epsilon\tilde{\bv}(\mu(\bq_0, q))
= \mu_2(q) + \epsilon\tilde{\bv}(\mu_2(q)). 
\]
Since $\mu_2$ is an automorphism, 
this defines a vector field on $O$. 
Formally,  
we associate to any $\bv\in k^d$ 
the constant vector field $(\varphi(\bv), 0)$ 
to $O^n\times O$ along $O':=\bq_0\times O$. 
By mapping it by the homomorphism $d\mu: T(O^n\times O)\to \mu^*(TO)$, 
we obtain a section of $\mu^*(TO)|_{O'}$. 
Since $\mu|_{Q'}=\mu_2$ is an isomorphism, 
it defines a vector field $\tilde{\bv}$ on $O$. 
From $\tilde{\bv}(q')=T_{\bq_0}\mu_1\circ\varphi(\bv)$, 
it follows that $\tilde{\be}_1, \dots, \tilde{\be}_d$ form 
a local frame of $TO$ at $q'$. 

Now let $f\in \Autq(Q)$ be such that $f(x)=x$, 
i.e., $f(\bq_0)=\bq_0$, $T_{\bq_0}f\circ\varphi=\varphi$ and $f(q_0)=q_0$. 
Since $f$ is a quandle homomorphism, we have 
\begin{eqnarray*}
f(\mu(\bq_0 + \epsilon \varphi(\bv), q))
& = &
f(\mu(\bq_0, q)
+ \epsilon\tilde{\bv}(\mu(\bq_0, q))) \\
& = & 
f(\mu_2(q))
+ \epsilon (T_{\mu_2(q)}f)(\tilde{\bv}(\mu_2(q))) 
\end{eqnarray*}
and 
\begin{eqnarray*}
f(\mu(\bq_0 + \epsilon \varphi(\bv), q))
& = &
\mu(f(\bq_0 + \epsilon \varphi(\bv)), f(q)) \\
& = & 
\mu(f(\bq_0) + \epsilon T_{\bq_0}f(\varphi(\bv)), f(q)) \\
& = & 
\mu(\bq_0 + \epsilon \varphi(\bv), f(q)) \\
& = & 
\mu(\bq_0, f(q)) + \epsilon \tilde{\bv}(\mu(\bq_0, f(q))) \\ 
& = & 
f(\mu_2(q)) + \epsilon \tilde{\bv}(f(\mu_2(q))), 
\end{eqnarray*}
hence 
\[
(T_{\mu_2(q)}f)(\tilde{\bv}(\mu_2(q))) = \tilde{\bv}(f(\mu_2(q))). 
\]
Since $\mu_2$ is an automorphism, we conclude that  
$f_*(\tilde{\bv})=\tilde{\bv}$. 
From $f(q')=q'$, 
$f_*(\tilde{\be}_i)=\tilde{\be}_i$ for $i=1, \dots, d$ 
and the fact that $\tilde{\be}_1, \dots, \tilde{\be}_d$ form a local frame at $q'$, 
it follows that $f$ is the identity morphism on the formal neighborhood of $q'$, 
hence on $Q$. 
\end{proof}

\begin{proof}[Proof of Proposition \ref{prop_good_action_to_algebraicity}]
(1)
By Proposition \ref{prop_inn_trans_opensub}, 
we have injective homomorphisms 
$\Inn(U)\to\Inn(Q)$ and $\Trans(U)\to\Trans(Q)$. 
One can identify $\Trans(U)$ and $\Trans(Q)$
with $\Trans(U, X)x$ and $\Trans(Q, X)x$, 
and they are equal by Proposition \ref{prop_orbit_opensub}. 
So $\Trans(U)\cong \Trans(Q)$. 
Since $\Inn(Q)$ is generated by $\Trans(Q)$ and any one of $s_q$, 
we have $\Inn(U)\cong\Inn(Q)$. 

\medskip
(2) 
Write $G$ for $\bTrans(Q)$. 
For a positive integer $n$, 
write $\ba$ for an $n$-ple $(a_1, \dots, a_n)$, etc., and 
\[
s_{\ba, \bb} = 
s_{b_1}^{-1}\circ\cdots\circ s_{b_n}^{-1}
\circ s_{a_n}\circ\cdots s_{a_1}
\]
for the automorphism of $G$ (as a variety) or $Q$, 
depending on the context. 
Define 
\[
\tau_n: Q^{2n}\to G;
\ (\ba, \bb)\mapsto s_{\ba, \bb}(x) 
\]
and 
\[
\alpha_n: Q^{2n}\times Q\to G\times Q\times Q; 
\ (\ba, \bb, q)
\mapsto 
(\tau_n(\ba, \bb), q, s_{\ba, \bb}(q)). 
\]
By Proposition \ref{prop_trans_orbit_in_x}, 
$\tau_n$ is surjective for $n\gg 0$. 
It follows that $\Img\ \alpha_n$ is the graph of the action map 
under the identification of $\Trans(Q)$ and $G$. 
We show that $\Img\ \alpha_n$ is a subvariety 
and that the action map is a regular map. 

Note that one can always replace $n$ by a larger number. 
We use the following two lemmas. 

\begin{lemma}
For $n\gg 0$ and any $g\in G$, 
there exists a point $(\ba, \bb)\in (Q_\sm)^{2n}$ 
such that $\tau_n(\ba, \bb)=g$ and $T_{(\ba, \bb)}\tau_n$ is surjective. 
\end{lemma}
\begin{proof}
By the surjectivity of $\tau_n$, 
there exists $(\ba_0, \bb_0)\in (Q_\sm)^{2n}$ 
such that $T_{(\ba_0, \bb_0)}\tau_n$ is surjective. 
Let $g_0=\tau_n(\ba_0, \bb_0)$. 
Since $G$ is a $\Trans(Q)$-orbit 
and $\Trans(Q)=\Trans(Q_\sm)$ 
by Proposition \ref{prop_good_action_to_algebraicity} (1), 
there exist $n'\in\bN$ and $(\ba_1, \bb_1)\in (Q_\sm)^{2n'}$ 
such that $s_{\ba_1, \bb_1}(g_0)=g$. 
Let $\tau': Q^{2n+2n'}\to G$ be defined by 
$\tau'(\ba, \bb, \ba', \bb')=s_{\ba', \bb'}\circ s_{\ba, \bb}(x)$. 
Then $\tau'(\ba_0, \bb_0, \ba_1, \bb_1)=g$ 
and $T_{(\ba_0, \bb_0, \ba_1, \bb_1)}\tau'$ is surjective. 

Using $s_a^{-1}\circ s_b = s_{a\rhd^{-1} b}\circ s_a^{-1}$, 
we see that 
there exists an automorphism $f$ of $Q^{2n+2n'}$ as a variety 
such that $\tau'=\tau_{n+n'}\circ f$. 
Thus $T\tau_{n+n'}$ is surjective at some smooth point over $g$. 

By openness of the smooth locus and noetherian induction, 
one can find $n$ that works for any $g\in G$. 
\end{proof}

\begin{lemma}
Let $S$ and $U$ be normal varieties, $T$ a variety, 
and $f: S\to T$ and $g: T\to U$ morphisms. 
Assume that $g\circ f$ is smooth and surjective 
and that $g$ is birational. 
Then $g$ is an isomorphism. 
\end{lemma}
\begin{proof}
Let $u\in U$ be an arbitrary point, 
and choose $s\in S$ such that $g(f(s))=u$. 
Let $T'$ be the normalization of $T$ 
and $f': S\to T'$ and $g': T'\to U$ the induced morphisms. 

We claim that $g'$ is \'etale at $f'(s)$. 
We have homomorphisms 
\[
\hat{\cO}_{U, u}\to \hat{\cO}_{T', f'(s)}\to \hat{\cO}_{S, s}. 
\]
Since $g\circ f$ is smooth, 
$\hat{\cO}_{S, s}$ can be identified with 
the power series ring over 
$\hat{\cO}_{U, u}$ with generators $h_1, \dots, h_d$. 
So we have homomorphisms 
\[
\hat{\cO}_{U, u}\overset{\varphi}{\to} \hat{\cO}_{T', f'(s)}
\overset{\psi}{\to} \hat{\cO}_{U, u}\cong \hat{\cO}_{S, s}/(h_1, \dots, h_d), 
\]
where $\psi\circ\varphi=id$. 
Then $\psi$ is surjective. 
Since $S, T'$ and $U$ are normal, 
all these rings are integral domains of dimension $\dim U$(\cite{Z1948}), 
so $\psi$ is also injective. 
It follows that $\varphi$ is an isomorphism, hence the claim. 

This means that $f'(S)$ is contained in the \'etale locus $V$ of $g'$, 
so $g'|_V$ is surjective. 
By \cite[Lemma 18.10.18]{Gr1967}, $g'|_V$ is an open immersion. 
Thus it is an isomorphism. 
It follows that $V=T'$ and $T=T'$, 
so $g$ is an isomorphism. 
\end{proof}

Let us return to the proof of (2). 
Let 
$S=Q^{2n}\times Q$, $T=\overline{\Img\ \alpha_n}$,  $U=G\times Q$, 
$f=\alpha_n$ and $g=(p_1\times p_2)|_T$, 
where $p_i$ is the projection from $G\times Q\times Q$ 
to the $i$-th factor. 
The restriction of $g$ to $\Img\ \alpha_n$ is bijective 
and $\Img\ \alpha_n$ is a dense constructible subset of $T$. 
This implies that $g: T\to U$ is birational, 
since we are working in characteristic $0$. 
We also know that any point of $U$ is in the image 
of the smooth locus of $g\circ f$. 

By the previous lemma, 
$g$ is an isomorphism. 
It also follows that $\Img\ \alpha_n$ is equal to $T$. 
Thus the action map, which is equal to $p_3\circ g^{-1}: G\times Q\to Q$, 
is a morphism. 
We can see that $\alpha^{-1}$ is also a morphism by the same arguments. 

(3)
If $f: S\times Q\to S\times Q$ is an algebraic family of transvections, 
then it induces an automorphism $\Phi(S)(f): S\times X\to S\times X$ 
over $S$ by Condition \ref{condition_good_action} (1). 
By sending the constant section $S\times \{x\}$ by this morphism 
and projecting to $X$, 
we obtain a morphism $\phi_f: S\to X$. 
It is straightforward to see that its image is contained in $\bTrans(Q)$ 
and that the pullback of $\alpha$ by $\phi_f$ is equal to $f$. 

(4)
This follows from (3) since, as is well known, 
an object representing a group functor 
is a group object in a natural way. 
For example, to show that the multiplication map is a morphism, 
let $\alpha_1, \alpha_2\in \Aut(\bTrans(Q)\times \bTrans(Q)\times Q)$ 
be defined by $(g_1, g_2, q)\mapsto (g_1, g_2, g_1q)$ 
and $(g_1, g_2, q)\mapsto (g_1, g_2, g_2q)$. 
Then $\alpha_1\circ\alpha_2$ is an element 
of $\cTrans(Q)(\bTrans(Q)\times\bTrans(Q))$, 
so it corresponds to a morphism $\bTrans(Q)\times\bTrans(Q)\to\bTrans(Q)$, 
which is nothing but the multiplication map. 

(5)
We have an element 
$(id\times s_q)\circ\alpha\circ (id\times s_r^{-1})$ 
of $\cTrans(Q)(\bTrans(Q))$, 
so it induces a morphism $\bTrans(Q)\to\bTrans(Q)$. 
\end{proof}

\begin{remark}
The group $\Inn(Q)$ is not necessarilly an algebraic group. 
For example, if $Q=\bA^1$ and 
$x\rhd y=x+a(y-x)$ with $a\not=0$, 
then $Q$ is $\rhd$-connected if and only if $a\not= 1$. 
It is easy to see that $\Inn(Q)$ consists of automorphisms 
$x\mapsto a^n x+b$ ($n\in\bZ, b\in k$), 
so it does not have a natural structure of an algebraic group 
unless $a$ is a root of unity. 
\end{remark}

As for non-$\rhd$-connected quandles, 
the following is known 
in the setting of differential geometry. 
\begin{theorem}[\cite{Loos1967}]\label{theorem_loos}
Let $(M, \cdot)$ be a reflexion space. 
Then there exists a regular $\varphi$-space $(G, H, \varphi)$ 
with $\varphi$ involutive 
and a manifold $F$ with an $H$-action 
such that $(M, \cdot)$ is isomorphic to $G\times_H F$ 
with the operation induced from $\cdot_\varphi$. 
\end{theorem}
Note that the orbits are $\rhd$-connected in this case. 
It might be possible to obtain 
a similar result for quandle varieties 
with this property.

\section{Questions related to $\rhd$-connectedness}

Kowalski asked the following problem. 
\begin{problem}(\cite[Problem II.47]{Kowalski1980})
Can one obtain any reasonable theory of 
``regular $s$-manifolds'' starting from the following conditions? 
\begin{enumerate}
\item
$x\cdot x=x$. 
\item
$s_x: M\to M; y\mapsto x\cdot y$ is a diffeomorphism. 
\item
$x\cdot(y\cdot z)=(x\cdot y)\cdot(x\cdot z)$. 
\item[(4I)]
$x$ is an isolated fixed point of $s_x$. 
\end{enumerate}
\end{problem}

We give a partial answer to its analogue in algebraic geometry. 

\begin{proposition}
Let $Q$ be a quandle variety 
such that $q$ is an isolated fixed point of $s_q$ for any $q\in Q$. 
Then the following hold. 

(1)
$Q\rhd q$ is dense for a general point $q$, 
hence $Q$ contains an open orbit $O$. 

(2)
For a point $q\in O$, 
the stabilizer $\bTrans(O)_q$ contains $(\bTrans(O)^{\varphi_q})^\circ$, 
so $O$ is an algebraic $\varphi$-space. 

(3)
The quandle homomorphism $\bar{s}: Q\to (\bTrans(Q), \rhd'_{\varphi_q}); x\mapsto s_xs_q^{-1}$ 
of Remark \ref{rem_vedernikov_hom} 
is quasi-finite. 
\end{proposition}
\begin{proof}
(1)
Define $f: Q\times Q\to Q\times Q$ by $f(q, r)=(q\rhd r, r)$. 
Let $\Delta\subseteq Q\times Q$ be the diagonal set 
and denote by $p_1, p_2$ the projection maps. 
Then, since $(q, q)=\Delta\cap p_1^{-1}(q)$ is isolated 
in $f^{-1}\Delta\cap p_1^{-1}(q)$, 
$\Delta$ is an irreducible component of $f^{-1}\Delta$. 
It follows that there exists a point $r$ 
such that $(r, r)=\Delta\cap p_2^{-1}(r)$ is isolated 
in $f^{-1}\Delta\cap p_2^{-1}(r)$. 
If we define $t_r: Q\to Q$ by $t_r(q):=q\rhd r$, 
this means that $r$ is isolated in $t_r^{-1}(r)$. 
Then $Q\rhd r$ is dense and the orbit $\Inn(Q)r$ is open. 

(2)
If $s_q g s_q^{-1} = g$ holds 
for an automorphism $g$ of $Q$, 
then $s_q = s_{g(q)}$ holds. 
In particular, $g(q)$ is a fixed point of $s_q$. 
If $g$ is contained in $(\bTrans(O)^{\varphi_q})^\circ$, 
then we must have $g(q)=q$ 
by the assumption that $q$ is an isolated fixed point of $s_q$. 

(3)
Similarly, 
if a fiber of $\bar{s}$ 
were to contain a positive dimensional variety $Z$, 
then $Z$ would be contained in the fixed point locus of $s_z$ 
for any $z\in Z$, a contradiction. 
Thus $\bar{s}$ is quasi-finite. 
\end{proof}

How various regularity conditions are related 
seems to be a subtle question. 
\begin{problem}
What implications hold between the following conditions? 
\begin{itemize}
\item[(T)]
$Q$ is nonsingular and $T_qs_q-1$ is invertible for any $q\in Q$. 
\item[(I)]
$q$ is an isolated fixed point of $s_q$ for any $q$. 
\item[(D)]
$Q\rhd q$ is dense for any $q$. 
\item[(C)]
$Q$ is $\rhd$-connected. 
\item[($\Phi$)]
$Q$ is an algebraic $\varphi$-space. 
\end{itemize}
\end{problem}

We know the following implications. 
\begin{itemize}
\item
(T) $\Rightarrow$ (I), (D), (C) and ($\Phi$). 
\item
(D) $\Rightarrow$ (C). 
\item
(I) and (C) $\Rightarrow$ (D) and ($\Phi$). 
\end{itemize}
As we saw in Remark \ref{rem_unip}, 
the condition (D) does not imply condition (I). 
The same example also shows that (D) and ($\Phi$) do not imply (I). 
In fact, 
let $G:=\SL_2(k)$ act on $Q$ by conjugation 
and define an automorphism $\varphi$ of $G$ by $M\mapsto J_1^{-1}MJ_1$. 
Then $G_{J_1}$ and $G^\varphi$ are both equal to 
$\{M\in\SL_2(k)| MJ_1=J_1M\}$, 
and $Q$ is isomorphic to $(G/G^\varphi, \rhd_\varphi)$. 

\medskip
As a related question, 
what happens if $T_qs_q=1$ for any $q\in Q$? 
Does there exist a $\rhd$-connected quandle variety $Q$ 
with $T_qs_q=1$ for any $q\in Q$? 
The following example shows 
that the orbits can be of codimension $1$.

\begin{example}
In this example, 
$Q$ is a weak algebraic $\varphi$-space, 
$T_qs_q=1$ for any $q\in Q$ 
and $\dim (\Inn(Q)q)=n-1$ 
where $n=\dim Q$. 
More precisely, 
the dimension of $Q^k q$ is $k$ for $k\leq n-2$ 
and $n-1$ for $k\geq n-1$. 

In Example \ref{ex_ext_of_trivial}, 
let $X=\bA^1$, $A=\bA^{n-1}$ and $F(x, y)=((y-x)^2, (y-x)^3, \dots, (y-x)^n)$. 
Then $Q=\bA^n$ with 
\[
(x_1, \dots, x_n)\rhd (y_1, \dots, y_n)=
(y_1, y_2+(y_1-x_1)^2, y_3+(y_1-x_1)^3, \dots, y_n+(y_1-x_1)^n). 
\]
Then 
$T_qs_q=1$ everywhere, 
but the orbit of $(0, 0)$ is $\{0\}\times \bA^{n-1}$. 

To see that $Q$ is a weak $\varphi$-space, 
note that the map $f_{\alpha_1, \alpha_2(x_1), \dots, \alpha_n(x_1)}: Q\to Q$ 
defined by 
\[
(x_1, \dots, x_n)\mapsto
(x_1+\alpha_1, 
x_2+\alpha_2(x_1), 
\dots, 
x_n+\alpha_n(x_1))
\]
is an automorphism of $Q$ 
for any $\alpha_1\in k$ and $\alpha_i(x_1)\in k[x_1]$, 
and that 
\[
G=\{f_{\alpha_1, \alpha_2(x_1), \dots, \alpha_n(x_1)} 
| \deg \alpha_i(x_1)\leq i-1 (i=2, \dots, n)\} 
\]
is a subgroup of $\Aut(Q)$ acting transitively on $Q$. 
We also see that the conjugation by $s_{(0, 0)}$ 
leaves $G$ invariant, 
and therefore defines an automorphism $\varphi$ of $G$. 
Then one can show that $Q$ is isomorphic to $(G/G_{(0, 0)}, \rhd_\varphi)$. 
\end{example}

\end{document}